\documentclass{amsart}
\usepackage{amssymb}
\usepackage[active]{srcltx}
\usepackage{a4wide}
\usepackage{amsmath,amsthm}
\usepackage{amssymb}
\usepackage{amstext}
\usepackage{cite}
\usepackage{epsfig}
\usepackage{enumerate}
\usepackage{color}
\usepackage{amssymb}
\newtheorem{theorem}{Theorem}[section]
\newtheorem{lemma}[theorem]{Lemma}

\theoremstyle{definition}
\newtheorem{definition}[theorem]{Definition}

\theoremstyle{remark}
\newtheorem{remark}[theorem]{Remark}
\newtheorem{note}[theorem]{Note}

\numberwithin{equation}{section}

\newcommand{\rectangle}{{%
  \ooalign{$\sqsubset\mkern3mu$\cr$\mkern3mu\sqsupset$\cr}%
}}

\newcommand{\N}{\mathbb{N}}
\newcommand{\R}{\mathbb{R}}
\newcommand{\cC}{\mathcal{C}}

\newcommand{\be}{\begin{equation}}
\newcommand{\ee}{\end{equation}}

\begin{document}

\title{On bivariate fractal approximation}

\author{V. Agrawal}
\address{Department of Mathematical Sciences, Indian Institute of Technology (BHU), Varanasi- 221005, India }
\email{vishal.agrawal1992@gmail.com}

\author{T. Som}
\address
{Department of Mathematical Sciences, Indian Institute of Technology (BHU), Varanasi- 221005, India}
\email{tsom.apm@iitbhu.ac.in}
\author{S. Verma}
\address{Department of Mathematics, Indian Institute of Technology Delhi, New Delhi- 110016, India }
\email{saurabh331146@gmail.com}


%\thanks{The author was supported in part by UGC.}

%    General info

%\date{January 1, 2001 and, in revised form, June 22, 2001.}

 

\begin{abstract}
In this paper, the notion of dimension preserving approximation for real -valued bivariate continuous functions, defined on a rectangular domain $\rectangle$, has been introduced and several results, similar to well-known results of bivariate constrained approximation in terms of dimension preserving approximants, have been established. Further, some clue for the construction of bivariate dimension preserving approximants, using the concept of fractal interpolation functions, has been added. In the last part, some multi-valued fractal operators associated with bivariate $\alpha$-fractal functions are defined and studied.
\end{abstract}

\maketitle

%%%%%%%%%%%%%%%%%%%%%%%%%%%%%%%%%%%%%%%%%%%%%%%%%%%%%%%%%%%%%%%%%%%%%%%%

%%%%%%%%%%%%%%%%%%%%%%%%%%%%%%%%%%%%%%%%%%%%%%%%%%%%%%%%%%%%%%%%%%%%%%%%
.

\section{Introduction} 
Following the seminal work of Barnsley \cite{MF1}, Navascu\'es \cite{M2,M1} studied the approximation of functions using their fractal counterparts termed as $\alpha$-fractal functions. In the same vein, 
Verma and Masspoust \cite{VM} recently introduced the notion of dimension preserving approximation. We use dim and $Gr(f)$ respectively to represent fractal dimension and graph of a function of $f$. \par
 Various concepts of fractal dimensions are available but we cover only those fractal dimensions that are suitable for this article. 
 We only need to mention the Hausdorff dimension, the box dimension, and the packing dimension defined for nonempty subsets of $\R^n$, $n\in \N$, and denoted by $\dim_H,~\dim_B$ and $\dim_P $ respectively.
 To know these fractal dimensions readers are suggested to go through, for instance, \cite{Fal,PM1}.

The following relations are established between these fractal dimensions. (see \cite{Fal}):
\[
\dim_H F \leq \underline{\dim}_B F \leq \overline{\dim}_B F
\]
and
\[
\dim_H F \leq \dim_P F  \leq \overline{\dim}_B F.
\]
\par
The class of all real-valued continuous functions on $ \rectangle:=I \times J$ is defined by  $ \cC\big(\rectangle\big)$ where $I=[a,b ]$ and $J=[c,d] .$   
 
 For a bivariate function $f$, we denote the derivative of $(k,l)$-th order by $D^{(k,l)}f$, that is, $D^{(k,l)}f:= \dfrac{\partial^{k+l}f}{\partial x^k \partial y^l}$. Let
     $$\mathcal{C}^{m,n}(\rectangle)= \{f: \rectangle \to \mathbb{R}; ~ D^{(k,l)}f \in  \cC\big(\rectangle\big),~~ \forall~~ 0\le k \le m, ~0\le l\le n \}.$$
     If $D^{(k,l)}f(\boldsymbol{x}) \ge 0, ~\forall~~ \boldsymbol{x} \in \rectangle,$ then we say the function $f$ is $(m,n)$-convex.
     Let $g\in \cC\big(\rectangle\big)$ such that $  \dim (Gr(g)) > 2$. We may refer to \cite{Shen} for the existence of such functions. The function $f:\rectangle \to \mathbb{R}$ defined by $f(x,y) :=\int\limits_{a}^{x}\int\limits_{c}^{y} g(t,s)dt ds$ satisfies the following: 
     \[
     \dim (Gr(f)) =2\quad\text{and}\quad\dim Gr(D^{(1,1)}f) =\dim (Gr(g)) > 2,
     \]
     where $\dim$ denotes a fractal dimension.
     
     Recall that the tensor product Bernstein polynomial on $\rectangle$ is defined as:
    $$B_{m,n}(f)(x,y)= \sum_{i=0}^m \sum_{j=0}^n f\Big(a+\frac{i(b-a)}{m}, c+\frac{j(d-c)}{n}\Big) {m \choose i} {n \choose j} (x-a)^i (b-x)^{m-i} (y-c)^j (d-y)^{n-j}.$$
    
Let us approximate a function $f \in \mathcal{C}^{k,l}(\rectangle)$ by $B_{m,n}(f)$, then (see \cite{Gal} for several properties of Bernstein polynomials) we have the following:
\begin{itemize}
\item $B_{m,n}(f) \to f$ uniformly on $ \rectangle.$
\item $\Big(D^{(k,l)}(B_{m,n}(f))\Big) \to D^{(k,l)}f$ uniformly on $ \rectangle.$
\item Since $B_{m,n}(f)$ and $D^{(k,l)}(B_{m,n}(f))$ are polynomials, then $\dim\Big(Gr\big(D^{(k,l)}(B_{m,n}(f))\big)\Big)=\dim(Gr(B_{m,n}(f)))=\dim(Gr(f))=2.$
\end{itemize}
The above items may conclude that the approximation by Bernstein polynomials maintains the smoothness of a function but not (necessarily) the dimensions of its partial derivatives. 

The present paper explores the approximation perspective relative to fractal dimension of a function and its partial derivatives.

The paper is structured as follows. In Section 1, we give a brief introduction and some preliminaries needed for the paper. In Section 2, we start to prove some results regarding dimension preserving approximation. In Section 3, we define some multi-valued mappings which are defined with the help of bivariate $\alpha$-fractal functions, and establish some properties of them.

\section{Dimension preserving approximation of bivariate functions}
Firstly, we mention the following result required for our paper:

\begin{lemma}[\cite{VM}, Lemma $3.1$]\label{lipdim}
Let $A \subset \mathbb{R}^m $ and $f,g:A \rightarrow \mathbb{R}^n$ be continuous functions. Then,  
\[
\dim_H (Gr(f+g)) = \dim_H (Gr(g))\quad\text{and}\quad\dim_P (Gr(f+g)) = \dim_P (Gr(g))
\]
provided that $f$ is a Lipschitz function.
 \end{lemma}

\begin{remark}
Note that the above lemma is also true for box dimensions.
\end{remark}

Let us denote the class of $Y$-valued Lipschitz functions on $X$ by $\mathcal{L}ip (X,Y),$ where $(X,d_X)$ is a compact metric space and $(Y,\|.\|_Y)$ is a normed linear space. Note that this space is a dense subset of $\mathcal{C}(X,Y)$ with respect to the supremum norm.

 In view of Lipschitz invariance property of dimension, one may conclude that the upcoming theorem holds for all aforementioned dimensions.

\begin{theorem}\label{densethm}
Let $\dim(X) \leq \beta \leq \dim(X)+\dim(Y)$. Then the set  $\mathcal{S}_{\beta}:=\{f\in  \mathcal{C}(X,Y): \dim (Gr(f)) = \beta\}$ is dense in $\mathcal{C}(X,Y).$
\end{theorem}
\begin{proof}
   Let $f\in \mathcal{C}(X,Y)$ and $\epsilon>0.$ Using the density of $\mathcal{L}ip(X,Y)$ in $\mathcal{C}(X,Y)$, there exists $g$ in $\mathcal{L}ip (X,Y)$ such that $$ \|f-g\|_{\infty,Y} < \frac{\epsilon}{2}.$$ Further, we consider a non-vanishing function $h \in \mathcal{S}_{\beta}.$ Let $h_*= g +\frac{\epsilon}{2\|h\|_{\infty,Y}}h,$ which immediately gives $$\|g-h_*\|_{\infty,Y} \le  \frac{\epsilon}{2}.$$ This together with Lemma \ref{lipdim} implies that $\dim(Gr(h_*))=\dim(Gr(h))=\beta.$ Hence, we have $h_* \in \mathcal{S}_{\beta}$ and $$ \|f-h_*\|_{\infty,Y} \le  \|f-g\|_{\infty,Y} + \|g-h_*\|_{\infty,Y} < \epsilon .$$ Thus, the proof of the theorem is complete.
\end{proof}   
To the best our knowledge, the univariate version of the next theorem is well-known, however, we could not find a proof of the theorem in bivariate setting. Hence, we write a detailed proof of it. 

\begin{theorem}\label{Rudinthm}
             Let $\big(f_k\big)$ be a sequence of differentiable functions on $\rectangle$. Assume that for some $(x_0,y_0) \in \rectangle,$ the sequences  $\big(f_{k}(x_{0},.)\big)$ and $\big(f_{k}(.,y_0)\big)$ converges uniformly on $ [c, d]$ and $[a,b]$ respectively. If $(D^{(1,1)}f _k)$ converges uniformly on $\rectangle,$ then $\big(f_k\big)$ converges uniformly on $\rectangle$ to a function $f$, and
             $$D^{(1,1)}f(\boldsymbol{x})=\lim_{k \to \infty }D^{(1,1)} f_k(\boldsymbol{x}),$$ for every $\boldsymbol{x} \in \rectangle.$
             \end{theorem}
\begin{proof}
Let  $\epsilon>0$. Since $(D^{(1,1)}f _k)$ converges uniformly, there exists $N_1 \in \mathbb{N}$ such that
$$|D^{(1,1)} f_{k}(\boldsymbol{x})-D^{(1,1)} f_{m}(\boldsymbol{x})| < \frac{\epsilon}{4(b-a)(d-c)}, ~~\forall~\boldsymbol{x} \in \rectangle, ~k,m \ge N_1.$$

By the mean-value theorem, see, for instance, \cite[Theorem $9.40$]{Rudin}, we have
\begin{equation}\label{visha2}
 \begin{aligned}
 &\big|f_{k}(x+h,y+k)-f_{m}(x+h, y+k)-f_{k}(x+h,y)+f_{m}(x+h, y)-f_{k}(x,y+k)+f_{m}(x, y+k)\\ & +f_{k}(x,y)-f_{m}(x, y)\big| \\= ~ & hk~ \big|D^{(1,1)}(f_k-f_{m})(t,s)\big| \\    \leq ~ & hk \max_{(t,s)\in \rectangle}\big|D^{(1,1)}f_k(t,s)- D^{(1,1)}f_{m}(t,s)\big| \\   \leq ~  & \frac{\epsilon}{4(b-a)(d-c)}hk  \\    \leq ~ & \frac{\epsilon}{4}.  
 \end{aligned}
\end{equation}
By the hypothesis for $(x_{0},y_{0}) \in \rectangle,$
one can choose $ N_0 ~(>N_1)~\in \mathbb{N}$ such that 
    $$|f_k(x_{0},y)-f_{m}(x_{0},y)|< \frac{\epsilon}{4}  ~~\forall~ k,m\geq N_0$$
   and
    $$|f_k(x,y_{0})-f_{m}(x,y_{0})| < \frac{\epsilon}{4} ~~ \forall~ k,m\geq   N_0.$$ 
    
Now, using the above estimates and Equation \ref{visha2} we have    

\begin{equation*}
\begin{aligned}
 |f_k(x,y)-f_{m}(x,y)| \leq&  \frac{\epsilon}{4}+|f_k(x,y_{0})-f_{m}(x,y_{0})|+|f_k(x_{0},y)-f_{m}(x_{0},y)|\\&+|f_k(x_{0},y_{0})-f_{m}(x_{0},y_{0})| \\ < &  \frac{\epsilon}{4}+\frac{\epsilon}{4}+\frac{\epsilon}{4}+\frac{\epsilon}{4}\\ < & \epsilon, 
\end{aligned}
\end{equation*}
 for every $(x,y) \in \rectangle$ and $k,m\geq   N_0.$ This immediately confirms the uniform convergence of $(f_k).$ The rest part follows by routine calculations, hence omitted.

\end{proof}
             
 \begin{lemma}\label{newlem2}
 Let $f:I \rightarrow \mathbb{R} $ be a Lipschitz map and $g:J \rightarrow \mathbb{R}$ be a continuous function. A mapping $h: \rectangle \to \mathbb{R}$ defined by $$h(x,y)=f(x)+g(y),$$ then $$dim_H (Gr(h)) = \dim_H (Gr(g))+1.$$
 \end{lemma}
 \begin{proof}
 Proof follows by defining a bi-Lipschitz mapping from $Gr(h)$ to the set $\{(x,y,g(y)):x \in I,~ y \in J \}.$
 \end{proof}
  Here, let us recall some dimensional results for univariate functions. Mauldin and Williams \cite{RD} considered the following class of functions:
 $$W_{b}(x):=\sum_{n=-\infty}^{\infty}b^{-\alpha n}[\phi(b^{n}x+\theta_{n})-\phi({\theta_{n}})],$$ where $\theta_{n}$ is an arbitrary real number, $\phi$ is a periodic function with period one and $ b > 1,$ $0<\alpha<1.$ They showed that for a large enough  $b$ there exists a constant $C>0$ such that $\dim_H (Gr(W_{b})$ is bounded below by $2-\alpha-(C/\ln b).$

 Further, a significant progress in dimension theory of functions is contributed by Shen \cite{Shen} for the following class of functions:
 $$f^{\phi}_{\lambda,b}(x):=\sum_{n=0}^{\infty}\lambda^{n} \phi(b^{n}x)$$
 where $b\geq 2$ and $\phi$ is a real-valued, $\mathbb{Z}$-periodic, non-constant, $C^{2}$-function defined on $\mathbb{R}$. He proved that there exists a constant $K_{0}$ depending on $\phi$ and $b$ such that if $1< \lambda b <K_{0}$ then $$\dim_H (Gr(f^{\phi}_{\lambda,b})= 2+ \frac{\log\lambda}{\log b}.$$

 For $f \in \mathcal{C}^{1,1}(\rectangle),$ we get $\dim(Gr(f))=2.$ However, no conclusion can be drawn for dimensions of its partial derivatives. This is evident from the following example: let  Weierstrass-type nowhere differentiable continuous function $W:I  \to \mathbb{R}$ as in \cite{Shen} with $1 \le  \dim( Gr(W))\le  2$. Now, we define $h:\rectangle \to \mathbb{R}$ by $$h(x,y)= W(x)+y.$$ Here, by Lemma \ref{newlem2}, we obtain $2 \le  \dim( Gr(h))=\dim( Gr(W))+1 \le  3.$ Then for the function $f$ defined by $$f(x,y):=\int\limits_{a}^{x} \int\limits_{c}^{y}h(t,s)dtds,$$ we have $\dim (Gr(f)) =2$ and $2 \le  \dim (Gr(D^{(1,1)}f)) =\dim (Gr(h))\le 3.$
 \begin{theorem}\label{mainthm}
              Let $f \in \mathcal{C}^{1,1}(\rectangle)$ such that $\dim (Gr(D^{(1,1)}f)) =\beta$ for some $2 \le \beta \le 3.$ Then we have a sequence $(f_k)$ in $\mathcal{C}^{1,1}(\rectangle)$ such that $\dim (Gr(D^{(1,1)} f_k)) =\beta$ and $f_k \to f$ uniformly on $\rectangle.$ 
             \end{theorem}
             \begin{proof}
             In view of Theorem \ref{densethm}, there exists a sequence $(g_k)$ in $ \mathcal{C}(\rectangle)$ such that $\dim (Gr(g_k)) =\beta$ and $g_k \to D^{(1,1)}f$ uniformly on $\rectangle.$ Further, let us consider a function $f_k: \rectangle \to \mathbb{R}$ defined by $$f_k(x,y):=\int\limits_{a}^{x}\int\limits_{c}^{y} g_k(t,s)dtds.$$ Then $D^{(1,1)}f_k=g_k$ and $(D^{(1,1)}f_k) \to D^{(1,1)}f$ uniformly. Next, we have that $\big(f_k(a,y)\big) \to 0$ and $\big(f_k(x,c)\big) \to 0$ uniformly on $I$ and $J$ respectively. Now, Theorem \ref{Rudinthm} provides the proof.
             \end{proof}
      
      %The next theorem deals with both dimension preserving and shape %preserving approximation of a continuous function.
     % \begin{theorem}\label{mainthm1}
                   % Suppose $f$ is a continuously differentiable function with $\dim (Gr(\partial_{xy}f)) =\beta$ for some $2 \le \beta \le 3$ and $f(x,y)\ge 0, \forall x,y \in [0,1].$ Then there exists a sequence of continuously differentiable functions $(f_k)$ satisfying $\dim \big(Gr($D^{(1,1)}f_k)\big) =\beta$ and $f_k(x,y)\ge 0, \forall x \in [0,1]$, and $(f_k)$ converges uniformly to $f$. 
     % \end{theorem}
      %\begin{proof}
     % The proof uses arguments similar to those given in Theorems \ref{densethm} and \ref{mainthm}, and is omitted.
     % \end{proof}

    \begin{theorem}
    Let $f \in \mathcal{C}(\rectangle)$ with $f(\boldsymbol{x}) \ge 0 ~\forall~\boldsymbol{x} \in \rectangle.$ Then, for a given $\epsilon >0,$ there exists $g \in \mathcal{S}_{\beta}$ satisfying the following: $$g(\boldsymbol{x}) \ge 0 ~\forall~\boldsymbol{x} \in \rectangle ~\text{and}~ \|f-g\|_{\infty} < \epsilon.$$
    \end{theorem}
    \begin{proof}
    Let $\epsilon >0.$ Theorem \ref{densethm} yields an element $h \in  \mathcal{S}_{\beta}$ such that $$ \|f-h\|_{\infty} < \frac{\epsilon}{2}.$$ We define $$g(\boldsymbol{x}):=h(\boldsymbol{x})+ \frac{\epsilon}{2}, ~\forall~ \boldsymbol{x} \in \rectangle.$$ Then, by Lemma \ref{lipdim}, $g \in \mathcal{S}_{\beta},$ and by routine calculations, we get  $$g(\boldsymbol{x})=h(\boldsymbol{x})-f(\boldsymbol{x})+f(\boldsymbol{x})+\frac{\epsilon}{2} \ge -\|f-h\|_{\infty} +f(\boldsymbol{x})+ \frac{\epsilon}{2}> f(\boldsymbol{x}) \ge 0.$$
    Furthermore, one has $$\|f- g\|_{\infty} \le \|f- h\|_{\infty} +\|h- g\|_{\infty} < \epsilon,$$ hence the proof. 
    \end{proof}
    
    \begin{theorem}
    Let $f:\rectangle \to \mathbb{R}$ be a $(m,n)$-convex function such that $f(a,y)=f(x,c)=0, ~\forall~x \in I, ~y \in J.$ Then for $\epsilon >0,$ there exists $(m,n)$-convex function $g$ such that $D^{(m,n)}g  \in \mathcal{S}_{\beta}$ and $\|f-g\|_{\infty} < \epsilon.$
   \end{theorem}
   \begin{proof}
   Let $\epsilon >0.$ Using Theorem \ref{densethm}, there exists $h \in \mathcal{S}_{\beta}$ such that $\|D^{(m,n)}f-h\| < \frac{\epsilon}{(b-a)^m(d-c)^n}.$ 
   By choosing $$g(x,y):= \int_{a}^{x} \int_c^y \dots  \int_{a}^{x_{m-1}} \int_{c}^{y_{n-1}} h(x_m,y_n)dx_{m}dy_{n}\dots dx_1 dy_1,$$ we have 
   \[
   \|f-g\| = \sup_{(x,y) \in \rectangle} \Big\{\Big| f- \int_{a}^{x} \int_c^y \dots  \int_{a}^{x_{m-1}} \int_{c}^{y_{n-1}} h(x_m,y_n)dx_{m}dy_{n}\dots dx_1 dy_1\Big|\Big\} < \epsilon,
   \]
   proving the assertion.
   \end{proof}
   
    \begin{theorem}\label{BSOSA}
    Let $f \in \mathcal{C}(\rectangle).$ Then, for $\epsilon >0$ there exists $g \in \mathcal{S}_{\beta}$ such that $$g(\boldsymbol{x}) \le f(\boldsymbol{x})~ \forall ~\boldsymbol{x} \in \rectangle~ \text{and}~ \|f-g\|_{\infty} < \epsilon.$$
    \end{theorem}
    \begin{proof}
    Since $f \in \mathcal{C}(\rectangle)$ and $\epsilon >0$,
    Theorem \ref{densethm} generates a member $h \in \mathcal{S}_{\beta}$ such that $$\|f-h\|_{\infty} < \frac{\epsilon}{2}.$$ 
    Choose $g(\boldsymbol{x}) :=h(\boldsymbol{x})- \frac{\epsilon}{2}, ~~~\forall~~\boldsymbol{x} \in \rectangle.$ Then, $$g(\boldsymbol{x})=h(\boldsymbol{x})-f(\boldsymbol{x})+f(\boldsymbol{x})-\frac{\epsilon}{2} \le \|f-h\|_{\infty} +f(\boldsymbol{x})- \frac{\epsilon}{2} < f(\boldsymbol{x}).$$ Furthermore, $$\|f- g\|_{\infty} \le \|f- h\|_{\infty} +\|h- g\|_{\infty} < \epsilon,$$
    establishing the proof.
    \end{proof}  
    Now, we aim to show the existence of best one-sided approximation. Let $\beta \in  [2, 3],$ and define
    $$\mathcal{C}_{\beta}(\rectangle) := \{ f \in \mathcal{C}(\rectangle) : \overline{\dim}_B (Gr(f)) \le \beta\}.$$ In view of \cite[Proposition $3.4$]{Fraser}, recall that $\mathcal{C}_{\beta}(\rectangle)$ is a normed linear space.
    Let $\{g_1,g_2,\dots,g_n\}$ be a linearly independent subset of $\mathcal{C}_{\beta}(\rectangle).$ Further, for a bounded below and Lebesgue integrable function $f: \rectangle \rightarrow \mathbb{R}$, we define  $$\mathcal{Y}_{n}^{\beta}(f):= \Big\{h \in span\{g_1,g_2, \dots, g_n\}: h (\boldsymbol{x}) \le f(\boldsymbol{x}) ~\forall ~\boldsymbol{x} \in \rectangle  \Big\}.$$ Theorem \ref{BSOSA} guarantees the nonemptyness of $\mathcal{Y}_{n}^{\beta}(f).$
     A function $h_f \in \mathcal{Y}_{n}^{\beta}(f)$ is said to be a best one-sided approximation from below to $f$ on $\rectangle$ if $$ \int_{\rectangle} h_f(\boldsymbol{x})~ d\boldsymbol{x}  = \sup \Big\{\int_{\rectangle } h(\boldsymbol{x})~ d\boldsymbol{x}: h \in \mathcal{Y}_{n}^{\beta}(f) \Big\}.$$
     In a similar way, we define best one-sided approximations from above. We state the next theorem for one-sided approximation from below. Though a similar result can be proved in terms of one-sided approximation from above, see, for instance, \cite{Devore,VV2}. 
     \begin{theorem}
     For a bounded below and integrable function $f: \rectangle \rightarrow \mathbb{R}$, there exists a member in $\mathcal{Y}_{n}(f)$ of best one-sided approximant from below to $f$ on $\rectangle$.
     \end{theorem}
     \begin{proof}
     Let $(h_m)$ be a sequence in $\mathcal{Y}_{n}(f)$ such that
     \begin{equation}\label{osidedeqn}
      \int_{\rectangle}  h_m(\boldsymbol{x})~ d\boldsymbol{x} \to  A ~~~\text{as}~~m \to \infty,
     \end{equation}
     where $A = \sup \Big\{\int_{\rectangle } h(\boldsymbol{x})~ d\boldsymbol{x}: h \in \mathcal{Y}_{n}^{\beta}(f) \Big\}.$
      With an appropriate constant $M_*> 0,$ we have 
     \begin{equation*}
     \begin{aligned}
     \int_{\rectangle}  |h_m(\boldsymbol{x})|~ d\boldsymbol{x} \le&  \int_{\rectangle}  \Big|h_m(\boldsymbol{x})- \frac{A}{(b-a)(d-c)}\Big|~ d\boldsymbol{x}\\ & + \int_{\rectangle} \frac{A}{(b-a)(d-c)}~ d\boldsymbol{x} \le M_* ,
     \end{aligned}
     \end{equation*}
     where $I=[a,b]$ and $J=[c,d].$ 
       Since $\mathcal{Y}_{n}^{\beta}(f)$ is a subset of finite-dimensional linear space, the closed set of radius $M_*$ in $\mathcal{Y}_{n}^{\beta}(f)$ is compact. Therefore, there exist a subsequence $ (h_{m_k})$ and a function $h$ in $\mathcal{Y}_{n}^{\beta}(f)$ such that the sequence $(h_{m_k})$ converges to $h$ in $\mathcal{L}^1(\rectangle).$ Recall a basic functional analysis result that every norm is equivalent on a finite-dimensional linear space. Now, from the finite-dimensionality of  $\mathcal{Y}_{n}^{\beta}(f)$, it follows that the sequence $(h_{m_k})$ also converges to $h$ uniformly.
     Further, since $ h_m(\boldsymbol{x}) \le f(\boldsymbol{x}), ~\forall ~\boldsymbol{x} \in \rectangle,$ and $h_{m_k} \to h$ uniformly, we get $h(\boldsymbol{x}) \le f(\boldsymbol{x}), ~~ \forall~ \boldsymbol{x} \in \rectangle.$ Thus, $h \in \mathcal{Y}_{n}^{\beta}(f).$ Now, by (\ref{osidedeqn}), we have $$\int_{\rectangle}h(\boldsymbol{x})~ d\boldsymbol{x}=\lim_{k \to \infty} \int_{\rectangle}  h_{m_k}(\boldsymbol{x})~ d\boldsymbol{x} = A,$$
     completing the task.
     \end{proof}
    
           %\begin{proof}
        %Let $f \in \mathcal{C}(\rectangle)$ and $\epsilon >0$ be given.
        %Theorem \ref{densethm} provides a bivariate function $h \in \mathcal{D}_{\beta}$ such that  $$\|f-h\|_{\infty} < \frac{\epsilon}{2}.$$ 
        %Choose $g(\boldsymbol{x}) =h(\boldsymbol{x})- \frac{\epsilon}{2}, ~~~\forall~~(\boldsymbol{x}) \in \rectangle.$ Then, $$g(\boldsymbol{x})- g(t,y)=h(\boldsymbol{x})-f(\boldsymbol{x})+f(\boldsymbol{x})- h(t,y)-f(t,y)+f(t,y)-\frac{\epsilon}{2} \le \|f-h\|_{\infty} +f(\boldsymbol{x})- \frac{\epsilon}{2} < f(\boldsymbol{x}) .$$  Now, we have $g(\boldsymbol{x}) \le f(\boldsymbol{x})$ for every $(\boldsymbol{x}) \in \rectangle.$ Moreover,
        %Furthermore, $$\|f- g\|_{\infty} \le \|f- h\|_{\infty} +\|h- g\|_{\infty} < \epsilon.$$
        %Hence, we obtain a function $g \in \mathcal{D}_{\beta}$ satisfying the required conditions.
        %\end{proof}

\subsection{Construction of dimension preserving approximants}
  First, Hutchinson \cite{H} hinted at the generation of parameterized fractal curves. In \cite{MF1}, Barnsley introduced Fractal Interpolation Functions (FIFs) via Iterated Function System (IFSs). It is important to choose IFS appropriately that it is fitted as an attractor for a graph of a continuous function called FIF. We refer to the reader \cite{MF1} for more study regarding the construction of FIFs.

Computation of dimensions of fractal functions has been an integral part of fractal geometry. In \cite{MF1}, Barnsley proved estimates for the Hausdorff dimension of an affine FIF. Falconer also established a similar results in \cite{Falc2}. Barnsley and his collaborators \cite{MF4,MF6,Hardin} computed the box dimension of classes of affine FIFs. In \cite{MF4}, FIFs generated by bilinear maps have been studied. In \cite{HM}, a formula for the box dimension of FIFs $\R^n\to\R^m$ was proved. A particular case of FIFs given by Navascu\'es \cite{M2}, namely, (univariate) \emph{$\alpha$-fractal function} has been proven very useful in approximation theory and operator theory. Using series expansion, the box dimension of (univariate) $\alpha$-fractal function is estimated in \cite{VV3}.
\par
Let us recall a construction of bivariate $\alpha$-fractal function introduced in \cite{VV1}, which was influenced by Ruan and Xu \cite{Ruan}, on rectangular grids.\\
Let $x_0=a,~x_N=b,~y_0=c,~y_M=d,$ and $f \in \mathcal{C}(\rectangle).$ Let us denote $\Sigma_k=\{1,2,\dots,k\},$ $ \Sigma_{k,0}=\{0,1,\dots k \},$ $\partial \Sigma_{k,0}=\{0,k\} $ and int$\Sigma_{k,0}=\{1,2,\dots,k-1\}.$ Further, a net $\Delta$ on $\rectangle$ is defined as follows:
$$ \Delta:=\{(x_i,y_j):i \in \Sigma_{N,0},~j \in \Sigma_{M,0}~ \text{and}~ x_0<x_1<\dots<x_N; ~y_0<y_1<\dots<y_M\}.$$
For each $i \in \Sigma_N$ and $j \in \Sigma_M$, let us define $I_i=[x_{i-1},x_i],~J_j=[y_{j-1},y_j]$ and $\rectangle_{ij}:=I_i \times J_j.$ Let $i \in \Sigma_N,$ we define contraction mappings $u_i:I \rightarrow I_i$ such that
$$ u_i(x_0)=x_{i-1}, ~~ u_i(x_N)=x_i, ~~\text{if $i$ is odd}, ~~\text{and}
        ~ u_i(x_0)=x_i,~~ u_i(x_N)=x_{i-1},~~ \text{if $i$ is even.}$$
Similar to the above, for each $j \in \Sigma_M,$ we define $v_j:J \rightarrow J_j,$ and $Q_{ij}(\boldsymbol{x}):= (u_i^{-1}(x),v_j^{-1}(y)),$ where $\boldsymbol{x}=(x,y) \in \rectangle_{ij}.$

Let $\alpha \in  \mathcal{C}(\rectangle)$ be such that $\|\alpha\|_{\infty}<1.$ Assume further that $s \in  \mathcal{C}(\rectangle)$ satisfying $s(x_i,y_j)=f(x_i,y_j),$ for all $i \in \partial \Sigma_{N,0}, j \in \partial \Sigma_{M,0}.$ By \cite[Theorem $3.4$]{VV2}, we have a unique function $f^{\alpha}_{\Delta,s} \in  \mathcal{C}(\rectangle)$ termed as $\alpha$-fractal function, such that 
 \begin{equation*}
   f^{\alpha}_{\Delta,s}(\boldsymbol{x})= f(\boldsymbol{x})+\alpha(\boldsymbol{x})~ f^{\alpha}_{\Delta,s}\big(Q_{ij}(\boldsymbol{x})\big)- \alpha(\boldsymbol{x})~s\big(Q_{ij}(\boldsymbol{x})\big),
\end{equation*}
for $\boldsymbol{x} \in \rectangle_{ij},~ (i,j) \in \Sigma_N \times \Sigma_M.$

\begin{note}
In this note, we recall Theorem $5.16$ in \cite{VV2}. 
With the metric $$d_{\rectangle}(\boldsymbol{x},\boldsymbol{y}):=\sqrt{(x_1-y_1)^2+(x_2-y_2)^2},~~\text{where}~\boldsymbol{x}=(x_1,x_2),~ \boldsymbol{y}=(y_1,y_2),$$ we consider $f$ and $s$ such that
\begin{equation}\label{Hypo}
\begin{aligned}
& |f(\boldsymbol{x}) -f(\boldsymbol{y})| \le K_f d_{\rectangle}(\boldsymbol{x},\boldsymbol{y})^{\sigma},\\&
|s(\boldsymbol{x}) -s(\boldsymbol{x})| \le K_s d_{\rectangle}(\boldsymbol{x},\boldsymbol{y})^{\sigma}.
\end{aligned}
\end{equation}
for every $\boldsymbol{x},\boldsymbol{y} \in \rectangle,$ and for fixed $K_f, K_s > 0.$ Assume that for some $k_f>0, \delta_0 >0$ the following holds: for each $\boldsymbol{x} \in  \rectangle $ and $ 0< \delta <\delta_0$ there exists $\boldsymbol{y}$ such that $d_{\rectangle}(\boldsymbol{x},\boldsymbol{y}) \le \delta$ and
\begin{equation} \label{HCeq2}
|f(\boldsymbol{x})-f(\boldsymbol{y})| \ge k_fd_{\rectangle}(\boldsymbol{x},\boldsymbol{y})^{\sigma}.
\end{equation} 
Furthermore, we suppose $N=M,~x_i-x_{i-1} = \frac{1}{N},~y_j-y_{j-1} =\frac{1}{M}, \forall ~ i \in \Sigma_N, ~j \in  \Sigma_M$ and constant scaling function $\alpha.$\\
If $ |\alpha|< \min\Big\{\frac{1}{M},\frac{k_f}{(K_{f^\alpha}+K_s)M^{\sigma}}\Big\},$ then $\dim_B\big(Gr(f^{\alpha})\big) = 3 - \sigma.$
\end{note}
\begin{remark}
With the assumptions in the above note, one may construct dimension preserving approximants for a given function, see, for instance, \cite[Theorem $3.16$]{VM}. 
\end{remark}
Navascu\'es \cite{M1} developed the notion of (univariate) $\alpha$-fractal function via so-called (univariate) fractal operator. In \cite{VV1,VV2}, her collaborators extended some of her results in bivariate setting. On putting $L= B_{m,n}$ in \cite[Theorem $3.1$]{VV1}, we have a unique function $f^{\alpha}_{\Delta,B_{m,n}} \in  \mathcal{C}(\rectangle)$ such that 
 \begin{equation}\label{Fnleq1}
   f^{\alpha}_{\Delta,B_{m,n}}(\boldsymbol{x})= f(\boldsymbol{x})+\alpha(\boldsymbol{x})~ f^{\alpha}_{\Delta,B_{m,n}}\big(Q_{ij}(\boldsymbol{x})\big)- \alpha(\boldsymbol{x})~B_{m,n}(f)\big(Q_{ij}(\boldsymbol{x})\big),
\end{equation}
for $\boldsymbol{x} \in \rectangle_{ij},~ (i,j) \in \Sigma_N \times \Sigma_M.$

Following the work of \cite{VV1}, we define a single-valued fractal operator 
$\mathcal{F}^\alpha_{m,n}: \mathcal{C}(\rectangle) \to \mathcal{C}(\rectangle)$ by $$\mathcal{F}_{m,n}^\alpha(f) =f^{\alpha}_{\Delta,B_{m,n}}.$$
In \cite{VV1}, several operator theoretic results for fractal operator are obtained. We recall that $\mathcal{F}^\alpha_{m,n}$ is a bounded linear operator, see, for instance, \cite[Theorem $3.2$]{VV1}.
\begin{lemma}[\cite{CC}, Lemma $1$]
	Let $(X,\|.\|)$ be a Banach space, $T: X \to X$ be a linear operator. Suppose there exist constants $\lambda_1, \lambda_2 \in [0,1)$ such that
	$$ \|Tx-x\| \le \lambda_1 \|x\| + \lambda_2 \|Tx\|, \quad \forall~~ x \in X.$$
	Then $T$ is a topological isomorphism, and
	$$\frac{1-\lambda_2}{1+\lambda_1} \|x\| \le  \| T^{-1}x\| \le \frac{1+\lambda_2}{1-\lambda_1}  \|x\|,\quad \forall~~x \in X.$$
\end{lemma}
\begin{note}\label{use7}
We have the following.
\begin{equation*}
\begin{split}
B_{m,n}(f)(\boldsymbol{x})= &\frac{1}{(b-a)^m(d-c)^n}\sum_{i=0}^m \sum_{j=0}^n {m \choose i} {n \choose j}(x-a)^i (b-x)^{m-i}\\& (y-c)^j (d-y)^{n-j}f\Big(a+\frac{i(b-a)}{m}, c+\frac{j(d-c)}{n}\Big),
\end{split}
\end{equation*}
Choosing $f=1,$ we have
\begin{equation*}
\begin{split}
 B_{m,n}1(\boldsymbol{x}) & = \frac{1}{(b-a)^m(d-c)^n}\sum_{i=0}^m \sum_{j=0}^n {m \choose i} {n \choose j}(x-a)^i (b-x)^{m-i}(y-c)^j (d-y)^{n-j}\\ &=\frac{1}{(b-a)^m(d-c)^n}\sum_{i=0}^m {m \choose i}(x-a)^i (b-x)^{m-i}\sum_{j=0}^n {n \choose j}(y-c)^j (d-y)^{n-j}\\ & =\frac{1}{(b-a)^m(d-c)^n}\sum_{i=0}^m {m \choose i}(x-a)^i (b-x)^{m-i}(y-c +d-y)^{n}\\ & =\frac{1}{(b-a)^m(d-c)^n}(x-a+b-x)^m (y-c +d-y)^{n} \\ & =1.
 \end{split}
 \end{equation*}
 This implies that $\|B_{m,n}\| \ge 1.$
 Now, for every $f \in \mathcal{C}(\rectangle)$ we get 
 \begin{equation*}
 \begin{split}
 |B_{m,n}(f)(\boldsymbol{x})| & \le  \frac{\|f\|_{\infty}}{(b-a)^m(d-c)^n}\sum_{i=0}^m \sum_{j=0}^n {m \choose i} {n \choose j}(x-a)^i (b-x)^{m-i}(y-c)^j (d-y)^{n-j}\\ & = \|f\|_{\infty},
 \end{split}
 \end{equation*}
 which produces $\|B_{m,n}\| \le 1.$ Therefore, we have $\|B_{m,n}\|=1.$
\end{note}
\begin{theorem}\label{thmtopiso}
	 The fractal operator $\mathcal{F}^\alpha_{m,n}:
	\mathcal{C}(\rectangle) \to \mathcal{C}(\rectangle)$ is a topological isomorphism.
\end{theorem}
\begin{proof}
	Using equation (\ref{Fnleq1}) and note \ref{use7}, one gets
	\begin{equation*}
	\begin{split}
\big	\| f - \mathcal{F}^\alpha_{m,n}(f)\big\|_\infty \le~ \|\alpha\|_\infty  \big\|\mathcal{F}_{m,n}^\alpha(f) -  B_{m,n}f  \big\|_\infty = ~\|\alpha\|_\infty \big\|\mathcal{F}_{m,n}^\alpha(f)\big\|_\infty+ \|\alpha\|_\infty \|f\|_\infty .
	\end{split}
	\end{equation*}
Since $\|\alpha\|_\infty < 1$, the previous lemma yields that the fractal operator $\mathcal{F}_{m,n}^\alpha$ is a topological isomorphism.
\end{proof}
\begin{remark}
The above theorem may strengthen item-4 of \cite[Theorem $3.2$]{VV1}. To be precise, item-4 tells that $\mathcal{F}_{m,n}^\alpha$ is a topological isomorphism if $ \|\alpha\|_\infty < \big(1+\|I - B_{m,n}\|\big)^{-1},$ which is more restricted than the standing assumption considered in the above theorem, that is, $\|\alpha\|_\infty < 1.$
\end{remark}
%\begin{example}
 %	Consider the following function in the domain $[-1,1] \times [-1,1]$ as the germ function.
 %	$$ f(\boldsymbol{x}) = 100(y-x^2)^2 +(1-x)^2,$$ 
 	%which is depicted in Fig.\ref{figonesided2}. Recall that the bivariate function $f$ defined above is the Rosenbrock's valley function in $\mathbb{R}^2,$ one of the benchmark functions in the field of global optimization. Consider the following choice of the parameters:
 	
 	%\begin{enumerate}[(i)]
 		%\item A net $\Delta$ determined by the partition $\{-1,-0.5,0,0.5,1\}$ of $[-1,1].$
 		
 		%\item Constant scaling function $\alpha(\boldsymbol{x}) = 0.9$ for all $(\boldsymbol{x}) \in [-1,1] \times [-1,1]$.

 		%\item Base function $s(\boldsymbol{x}) = x^2 y^{10} f(\boldsymbol{x}).$	
 	%\end{enumerate}

 %As in the case of two-dimensional Rosenbrock's function, it is felt that the ``fractalized Rosenbrock's functions" can find rich applications in optimization, although we do not claim a direct application herein. Further, the degrees of freedom offered by this fractalization procedure may be useful to deal with some problems combined with approximation and optimization.\end{example}
 
\begin{theorem}
Let $f \in \mathcal{C}(\rectangle)$ be such that $f(\boldsymbol{x}) \geq 0,~\forall ~\boldsymbol{x} \in \rectangle.$ Then for $\epsilon >0,$ and for $\alpha \in \mathcal{C}(\rectangle)$ satisfying $\|\alpha \|_{\infty} < 1,$ we have an $\alpha$-fractal function $ g_{\Delta,B_{m,n}}^{\alpha}$ satisfying $$ g_{\Delta,B_{m,n}}^{\alpha}(\boldsymbol{x}) \geq 0, ~~\forall ~\boldsymbol{x} \in \rectangle~~\text{and}~ \|f- g_{\Delta,B_{m,n}}^{\alpha}\|_{\infty} <\epsilon.$$
\end{theorem}
\begin{proof}
Note that the Bernstein operator $B_{m,n}$ fixes the constant function $1$, that is, $B_{m,n}(1)=1,$ where $ 1(\boldsymbol{x})=1 $ on $\rectangle.$  Consider $ \alpha \in \mathcal{C}(\rectangle)$ such that $\|\alpha\|_{\infty} < 1.$ From Equation \ref{Fnleq1}, we deduce
$$ \|g_{\Delta,B_{m,n}}^{\alpha}- g\|_{\infty} \leq \|\alpha \|_{\infty}\|g_{\Delta,B_{m,n}}^{\alpha}- B_{m,n}g\|_{\infty}, ~\forall~ g \in \mathcal{C}(\rectangle).$$ Choose $g=1$, then the above inequality gives $$ \|f_{\Delta,B_{m,n}}^{\alpha}- 1\|_{\infty} \leq \|\alpha \|_{\infty}\|f_{\Delta,B_{m,n}}^{\alpha}- 1\|_{\infty},$$ and this further yields $\|f_{\Delta,B_{m,n}}^{\alpha}- 1\|_{\infty} = 0.$ Therefore, $f_{\Delta,B_{m,n}}^\alpha=1$, that is, $\mathcal{F}_{m,n}^{\alpha}(1)=1.$\\
For $\epsilon >0 $, $\alpha \in \mathcal{C}(\rectangle)$ and $ f \in \mathcal{C}(\rectangle).$ Using Theorem \ref{densethm}, there exists a function $h_{\Delta,B_{m,n}}^\alpha$ such that $$ \|f- h_{\Delta,B_{m,n}}^{\alpha}\|_{\infty} <\frac{\epsilon}{2}, ~ \text{ where}~ \mathcal{F}_{m,n}^{\alpha}(h)=h_{\Delta,B_{m,n}}^{\alpha}.$$ Define
 $ g_{\Delta,B_{m,n}}^{\alpha}(\boldsymbol{x})=h_{\Delta,B_{m,n}}^{\alpha}(\boldsymbol{x})+ \frac{\epsilon}{2} $ for all $\boldsymbol{x} \in \rectangle.$ Since $\mathcal{F}_{m,n}^{\alpha}(1)=1,$   $$ g_{\Delta,B_{m,n}}^{\alpha}(\boldsymbol{x})=h_{\Delta,B_{m,n}}^{\alpha}(\boldsymbol{x})+ \frac{\epsilon}{2}1(\boldsymbol{x}) = h_{\Delta,B_{m,n}}^{\alpha}(\boldsymbol{x})+ \frac{\epsilon}{2}1^{\alpha}(\boldsymbol{x}).$$
 Further, since $\mathcal{F}_{m,n}^{\alpha}$ is a linear operator $$ g_{\Delta,B_{m,n}}^{\alpha} = h_{\Delta,B_{m,n}}^{\alpha}+ \frac{\epsilon}{2}1^{\alpha}= \mathcal{F}_{m,n}^{\alpha}(h+\frac{\epsilon}{2} 1).$$
 Moreover, 
 \begin{equation*}
 \begin{aligned}
 g_{\Delta,B_{m,n}}^{\alpha}(\boldsymbol{x}) & = h_{\Delta,B_{m,n}}^{\alpha}(\boldsymbol{x})+ \frac{\epsilon}{2}\\ & = h_{\Delta,B_{m,n}}^{\alpha}(\boldsymbol{x})+ \frac{\epsilon}{2}-f(\boldsymbol{x})+f(\boldsymbol{x}) \\ & \geq f(\boldsymbol{x})+ \frac{\epsilon}{2} - \| h_{\Delta,B_{m,n}}^{\alpha}- f \|_{\infty}\\&  \geq 0.
 \end{aligned}
 \end{equation*}
 Further, we get
 \begin{equation*}
  \begin{aligned}
 \| f-g_{\Delta,B_{m,n}}^{\alpha} \|_{\infty} & \leq \| f-h_{\Delta,B_{m,n}}^{\alpha} \|_{\infty}+\| h_{\Delta,B_{m,n}}^{\alpha}-g_{\Delta,B_{m,n}}^{\alpha} \|_{\infty}\\&< \frac{\epsilon}{2}+\frac{\epsilon}{2} \\&= \epsilon,
 \end{aligned}
  \end{equation*}
  completing the proof.
 \end{proof}

\section{Some multi-valued mappings}

First, we collect some definitions and related results which will be used in this section.
\begin{definition}(\cite{Aubin}).
  Let $(X,\|.\|_X)$ and $(Y,\|.\|_Y)$ be normed linear spaces. For a multi-valued (set-valued) mapping $T: X \rightrightarrows Y$, the
   domain of $T$ is defined by
  $\text{Dom}(T):= \{x \in X: T(x) \neq \emptyset\}.$ Then $T: X \rightrightarrows Y$ is
  \begin{enumerate}
  \item  \emph{convex}  if $$\lambda T(x_1)+(1-\lambda)T(x_2) \subseteq T\big(\lambda x_1+(1-\lambda)x_2\big), ~\forall ~x_1, x_2 \in \text{Dom}(T),~~\lambda \in [0,1].$$

  \item  \emph{process}  if $$\lambda T(x)=  T(\lambda x),~ \forall~x \in X,~\lambda > 0, ~\text{and}~ 0 \in T(0).$$

      \item \emph{linear} if $$\beta T(x_1)+ \gamma T(x_2) \subseteq T\big(\beta x_1+\gamma x_2\big), ~\forall ~ x_1, x_2 \in \text{Dom}(T),~\beta, \gamma \in \mathbb{R}.$$

          \item \emph{closed} if the graph of $T$ defined by $Gr(T):= \big\{(\boldsymbol{x})\in X \times Y: y \in T(x)  \big\}$ is closed.

          \item \emph{Lipschitz} if 
          $$T(x_1) \subseteq T(x_2) + l \|x_1-x_2\|_X ~U_Y,~\forall~x_1, x_2 \in \text{Dom}(T), ~\text{for some constant}~ l>0,$$
          where $U_Y=\{y\in Y: \|y\|_Y\le 1\}$.
         \item \emph{lower semicontinuous} at $x\in X$ if there exists a $\delta > 0$ such that $$U \cap T(x') \neq \emptyset ~\text{ whenever}~ \|x-x'\|_X < \delta$$ holds for a given open set $U$ in $Y$ satisfying $ U \cap T(x) \neq \emptyset .$
  \end{enumerate}
  \end{definition}
  Note that the above definitions are also applicable in metric spaces with obvious modifications, see, for instance, \cite{Aubin}.
  \begin{theorem}[\cite{DS}, Corollary $1.4$] \label{Multhm2}
  Let $T: \text{Dom}(T)=X \rightrightarrows Y$ be linear such that $T(0)= \{0\}.$ Then, $T$ is single-valued.
 \end{theorem}
  \begin{theorem}[\cite{DS}, Corollary $2.1$]\label{Multhm2a}
  Let $T: \text{Dom}(T)=X \rightrightarrows Y$ be such that $T(x_0)$ is singleton for some $x_0\in X.$ Then the following are equivalent:
  \begin{itemize}
  \item $T$ is single-valued and affine.
  \item $T$ is convex.
  \end{itemize}
  
  \end{theorem}
  Our work in this part is partly motivated by \cite{VV3}.
  \begin{theorem}\label{Multhm3}
  The multi-valued mapping $ \mathcal{W}^\alpha_{\Delta}: \mathcal{C}(\rectangle) \rightrightarrows \mathcal{C}(\rectangle)$ defined by $$\mathcal{W}^\alpha_{\Delta}(f) =\{f^{\alpha}_{\Delta,B_{m,n}}: m,~n \in \mathbb{N} \}$$ is a Lipschitz process.
  \end{theorem}
  \begin{proof}
   Using the linearity of $\mathcal{F}_{m,n}^\alpha,$ we have

   \begin{equation*}
                   \begin{aligned}
            \mathcal{W}^\alpha_{\Delta} (\lambda f) = \{(\lambda f)^{\alpha}_{\Delta,B_{m,n}}: m,~n \in \mathbb{N} \}
                   =  \lambda \mathcal{W}^\alpha_{\Delta}(f),~ \forall~f \in \mathcal{C}(\rectangle),~\lambda > 0.
\end{aligned}
   \end{equation*}
Again by linearity of $\mathcal{F}_{m,n}^\alpha,$ it is plain that $\mathcal{W}^\alpha_{\Delta} (0) = \{0\}.$ Therefore,  $\mathcal{W}^\alpha_{\Delta}$ is a process. 

\par
Let $f,g \in \mathcal{C}(\rectangle).$
 On applying Equation \ref{Fnleq1}, we have
    
 \begin{equation*}
           \begin{aligned}
                    \big|f^{\alpha}_{\Delta,B_{m,n}}(\boldsymbol{x})- g^{\alpha}_{\Delta,B_{m,n}}(\boldsymbol{x})\big| \le & ~ \|f- g\|_{\infty}+\|\alpha\|_{\infty}\|f^{\alpha}_{\Delta,B_{m,n}}- g^{\alpha}_{\Delta,B_{m,n}}\|_{\infty}\\&+\|\alpha\|_{\infty} \|B_{m,n}(g)- B_{m,n}(f)\|_{\infty},
          \end{aligned}
    \end{equation*}
 for any $\boldsymbol{x} \in \rectangle.$
 Further, we deduce
 $$\|f^{\alpha}_{\Delta,B_{m,n}}- g^{\alpha}_{\Delta,B_{m,n}}\|_{\infty} \le \frac{1+\|\alpha\|_{\infty}\|B_{m,n}\|}{1- \|\alpha\|_{\infty}} \|f-g\|_{\infty}.$$
 Using $\|B_{m,n} \| = 1,$
  $$\|f^{\alpha}_{\Delta,B_{m,n}}- g^{\alpha}_{\Delta,B_{m,n}}\|_{\infty} \le \frac{1+\|\alpha \|_{\infty}}{1- \|\alpha \|_{\infty}} \|f-g\|_{\infty}.$$
  Consequently, we have
  $$ \mathcal{W}^\alpha_{\Delta}(g) \subseteq \mathcal{W}^\alpha_{\Delta}(f) +\dfrac{1+ \|\alpha \|_{\infty}}{1-\|\alpha\|_{\infty}}~ \|f-g\|_{\infty}U_{\mathcal{C}(\rectangle)},$$
  proving the Lipschitzness of $\mathcal{W}^\alpha_{\Delta},$ and hence the proof.
\end{proof}

\begin{remark}
For the multivalued mapping $\mathcal{W}^\alpha_{\Delta}$, let us first note the following:
\begin{enumerate}
\item By linearity of $\mathcal{F}_{\Delta, B_{m,n}}^\alpha,$ we have $\mathcal{W}^\alpha_{\Delta} (0) = \{0\}.$

\item Since if $\alpha \neq 0$, $m \neq k$ then $f_{\Delta,B_{m,n}}^\alpha \neq f_{\Delta,B_{k,l}}^\alpha$, hence $ \mathcal{W}^\alpha_{\Delta}: \mathcal{C}(\rectangle) \rightrightarrows \mathcal{C}(\rectangle)$ is not single-valued.
\end{enumerate}
In view of the above items, Theorems \ref{Multhm2}-\ref{Multhm2a} produce that the mapping $ \mathcal{W}^\alpha_{\Delta}: \mathcal{C}(\rectangle) \rightrightarrows \mathcal{C}(\rectangle )$ is not convex.
\end{remark}

 \begin{theorem}
 Let a fixed net $\triangle$ and  $m,n \in \mathbb{N},$ the multivalued mapping $\mathcal{T}^{\Delta}_{m,n}: \mathcal{C}(\rectangle) \rightrightarrows \mathcal{C}(\rectangle )$ by $$ \mathcal{T}^{\Delta}_{m,n}(f) =\{f^{\alpha}_{\triangle,B_{m,n}}: \alpha \in  \mathcal{C}(\rectangle ) ~\text{such that}~\|\alpha\|_{\infty} < 1  \}$$ is a process. 
 \end{theorem}
 \begin{proof}
 Let $f \in \mathcal{C}(\rectangle)$ and $\lambda > 0,$
 \begin{equation*}
                 \begin{aligned} 
                 \lambda \mathcal{T}^{\Delta}_{m,n}(f)=&\lambda\{f^{\alpha}:\alpha \in  \mathcal{C}(\rectangle ) ~\text{such that}~\|\alpha\|_{\infty} < 1   \}\\
                 =&\{ \lambda f^{\alpha}:\alpha \in  \mathcal{C}(\rectangle ) ~\text{such that}~\|\alpha\|_{\infty} < 1  \}\\
                =& \mathcal{T}^{\Delta}_{m,n}(\lambda f).
                 \end{aligned}
 \end{equation*} 
  Moreover, Using linearity of fractal operator, we have $f^{\alpha}=0,$ whenever $f=0.$ That is, $0 \in \mathcal{T}^{\Delta}_{m,n}(0).$
  Therefore, $\mathcal{T}^{\Delta}_{m,n}$ is a process.

 \end{proof}
 %\begin{theorem}
  %For a fixed partition $\triangle$ and operator $L,$ the multivalued mapping %$T: C(I) \rightarrow C(I)$ by $$ T(f) =\{f^{\alpha}_{\triangle,B_{m,n}}: %|\alpha|_{\infty} \le q <1 \}$$ is closed. 
  %\end{theorem}
  %\begin{proof}
 % We show that $G_T=\{(f,g): g \in T(f), ~\forall f \in C(I)\}$ is closed. Let $(f_k,g_n) \in G_T$ be a sequence such that $(f_k,g_n) \rightarrow (f,g).$ Since $g_n \in T(f_k),$ we obtain a sequence of scale vectors $\alpha_n$ such that $g_n= f^{\alpha_n}_n.$ It is obvious that the set $\{\alpha \in \mathbb{R}^{N-1}: |\alpha|_{\infty} \le q \}$ is a compact subset. Hence, the sequence $(\alpha_n)$ has a convergent subsequence. From this, we deduce that $g = f^{\alpha_0}.$ 

  %\end{proof}

 \begin{remark}
 One may see that $\mathcal{T}^{\Delta}_{m,n}$ is not convex through the following lines. Let $f, g \in \mathcal{C}(\rectangle),$ 
   \begin{equation*}
                    \begin{aligned} 
                     \mathcal{T}^{\Delta}_{m,n}(f+g)=&\{( f + g )^{\alpha}:\|\alpha\|_{\infty} < 1\}\\ 
                     =&\{ f^{\alpha}+g^{\alpha}:\|\alpha\|_{\infty} < 1 \}\\
                    \subseteq &\{ f^{\alpha}+g^{\beta}:\|\alpha\|_{\infty} < 1,\|\beta\|_{\infty} < 1 \}\\
                    =&\{f^{\alpha}:\|\alpha\|_{\infty} < 1  \}+\{g^{\beta}:\|\beta\|_{\infty} < 1  \}\\
                   \subseteq & \mathcal{T}^{\Delta}_{m,n}(f)+\mathcal{T}^{\Delta}_{m,n}(g).
                    \end{aligned}
    \end{equation*}
 \end{remark}
 \begin{theorem}
      Let a fixed net $\triangle$ and $m,n \in \mathbb{N},$ the multivalued mapping $\mathcal{T}^{\Delta}_{m,n}: \mathcal{C}(\rectangle) \rightrightarrows \mathcal{C}(\rectangle)$ defined by $$ \mathcal{T}^{\Delta}_{m,n}(f) =\{f^{\alpha}_{\triangle,B_{m,n}}: \|\alpha\|_{\infty} \le q < 1 \}, $$ satisfies the following: $$ \|\mathcal{T}^{\Delta}_{m,n}\| \le 1 + \frac{q}{1-q}\|Id -B_{m,n} \|.$$ 
   \end{theorem}
   \begin{proof}
   We have 
   \begin{equation*}
     \begin{aligned} 
       \|\mathcal{T}^{\Delta}_{m,n}\|=& \sup_{f \in \mathcal{C}(\rectangle)} \frac{d(0, \mathcal{T}^{\Delta}_{m,n}(f))}{\|f\|_{\infty}}\\
       =& \sup_{f \in \mathcal{C}(\rectangle)}\inf_{f^{\alpha} \in \mathcal{T}^{\Delta}_{m,n}(f)} \frac{\| f^{\alpha}\|}{\|f\|}\\
       \le & \sup_{f \in \mathcal{C}(\rectangle)} \Big(1+ \frac{\|\alpha\|_{\infty}}{1-\|\alpha\|_{\infty}}\|Id - B_{m,n}\| \Big)\\
     \le & \sup_{f \in\mathcal{C}(\rectangle)} \Big(1+ \frac{q}{1-q}\|Id - B_{m,n}\| \Big)\\
       =&  1+ \frac{q}{1-q}\|Id - B_{m,n}\|,
       \end{aligned}
       \end{equation*}
    hence the proof.
   \end{proof}
 \begin{theorem}
   For a fixed net $\triangle$ and operator $L,$ the multivalued mapping $\mathcal{T}^{\Delta}_{m,n}:\mathcal{C}(\rectangle) \rightrightarrows \mathcal{C}(\rectangle)$ defined by $$ \mathcal{T}^{\Delta}_{m,n}(f) =\{f^{\alpha}_{\triangle,B_{m,n}}: \|\alpha \|_{\infty} < 1 \}$$ is lower semicontinuous. 
 \end{theorem}
 \begin{proof}
 Let $f \in \mathcal{C}(\rectangle),$ let $f^{\alpha} \in \mathcal{T}^{\Delta}_{m,n}(f)$ and a sequence $(f_k)$ in $\mathcal{C}(\rectangle)$ such that $f_k \to f.$ Since the fractal operator is continuous, we have $f^{\alpha}_k \rightarrow f^{\alpha}.$ It is clear that $f^{\alpha}_k \in \mathcal{T}^{\Delta}_{m,n}(f_k).$ Therefore, the result follows.
\end{proof}
 
  \begin{theorem}
  
Let $\triangle$ be a net of $\rectangle$ and $m,n \in \mathbb{N}.$ The multi-valued mapping $\mathcal{T}^{\Delta}_{m,n}:\mathcal{C}(\rectangle) \rightrightarrows \mathcal{C}(\rectangle)$ defined by $$ \mathcal{T}^{\Delta}_{m,n}(f) =\{f^{\alpha}_{\triangle,B_{m,n}}: \|\alpha\|_{\infty} \le q < 1 \}, $$ is Lipschitz.
  
  \end{theorem}
 \begin{proof}
 Let $f,g \in \mathcal{C}(\rectangle).$ 
 Equation (\ref{Fnleq1}) yields
        
 \begin{equation*}
 \begin{aligned}
  \big|f^{\alpha}_{\triangle,B_{m,n}}(\boldsymbol{x})- g^{\alpha}_{\triangle,B_{m,n}}(\boldsymbol{x})\big|  = & \|f- g\|_{\infty}+\|\alpha\|_{\infty} \|f^{\alpha}_{\triangle,B_{m,n}}- g^{\alpha}_{\triangle,B_{m,n}}\|_{\infty}\\  & + \|\alpha\|_{\infty} \|B_{m,n}g- B_{m,n}f\|_{\infty}, 
 \end{aligned}
 \end{equation*}
 for every $\boldsymbol{x} \in \rectangle.$
 Further, we deduce
 $$\|f^{\alpha}_{\triangle,B_{m,n}}- g^{\alpha}_{\triangle,B_{m,n}}\| \le \frac{1+\|\alpha\|_{\infty}\|B_{m,n}\|}{1- \|\alpha\|_{\infty}} \|f-g\|_{\infty}.$$
 Since $\|\alpha\|_{\infty} \le q $ and $\|B_{m,n}\|=1,$ we get
  $$\|f^{\alpha}_{\triangle,B_{m,n}}- g^{\alpha}_{\triangle,B_{m,n}}\| \le \frac{1+q}{1- q} \|f-g\|.$$
  Choosing $l= \frac{1+ q}{1-q}, $ we have 
  $$ \mathcal{T}^{\Delta}_{m,n}(g) \subset \mathcal{T}^{\Delta}_{m,n}(f) +l~ \|f-g\|_{\infty} U_{\mathcal{C}(\rectangle)},$$
  proving the assertion.
 \end{proof}
 
 \begin{theorem}
   For a fixed admissible scale vector $\alpha$ and $m,n \in \mathbb{N},$ the multivalued mapping $\mathcal{V}^{\alpha}_{m,n}:\mathcal{C}(\rectangle) \rightrightarrows \mathcal{C}(\rectangle)$ defined by $$ \mathcal{V}^{\alpha}_{m,n}(f) =\{f^{\alpha}_{\triangle,B_{m,n}}: \text{all possible net}~ \triangle  \}$$ is a process. 
   \end{theorem}
   \begin{proof}
     Let $f \in \mathcal{C}(\rectangle)$ and $\lambda > 0,$ then
      \begin{equation*}
                      \begin{aligned} 
                      \lambda \mathcal{V}^{\alpha}_{m,n}(f)= & \lambda\{f^{\alpha}_{\triangle,B_{m,n}}:\text{all possible net}~ \triangle \}\\
                      =&\{ \lambda f^{\alpha}_{\triangle,B_{m,n}}:\text{all possible net}~ \triangle \}\\
                     =&\{(\lambda f)^{\alpha}_{\triangle,B_{m,n}}:\text{all possible net}~ \triangle \}\\
                     =& \mathcal{V}^{\alpha}_{m,n}(\lambda f).
                      \end{aligned}
      \end{equation*}
       The third equality follows from the fact that the fractal operator $ \mathcal{F}^{\alpha}_{m,n}$ is a linear operator. 
       Moreover, using linearity of the fractal operator, we have $f^{\alpha}_{\triangle,B_{m,n}}=0,$ whenever $f=0.$ That is, $0 \in \mathcal{V}^{\alpha}_{m,n}(0).$
       Therefore, $\mathcal{V}^{\alpha}_{m,n}$ is a process.
   
   \end{proof}
  \begin{theorem}
    For a fixed admissible scale function $\alpha$ and $m,n \in \mathbb{N},$ the multivalued mapping $\mathcal{V}^{\alpha}_{m,n}$ is lower semicontinuous. 
    \end{theorem}
    \begin{proof}
     Let $f \in \mathcal{C}(\rectangle),$ let $f^{\alpha}_{\triangle,B_{m,n}} \in \mathcal{V}^{\alpha}_{m,n}(f)$ and a sequence $(f_k)$ converges to $f$ in $\mathcal{C}(\rectangle).$ Since the fractal operator is continuous, we have $(f_k)^{\alpha}_{\triangle,B_{m,n}} \rightarrow f^{\alpha}_{\triangle,B_{m,n}}.$ By definition of $\mathcal{V}^{\alpha}_{m,n},$ $(f_k)^{\alpha}_{\triangle,B_{m,n}} \in \mathcal{V}^{\alpha}_{m,n}(f_k).$ Hence, the lower semicontinuity of $\mathcal{V}^{\alpha}_{m,n}$ follows.
    
    \end{proof}
    
%\begin{definition}[\cite{Aubin}]
%Let $ T: (X,d) \to (Y,\rho) $ be a set-valued map between two metric spaces.
%\begin{enumerate}
%\item $T$ is called lower semicontinuous at $x\in X$ if for any open set $U$  in $Y$ such that $ U \cap T(x) \neq \emptyset $ there exists a $\delta > 0$ satisfying $U \cap T(x') \neq \emptyset $ whenever $d(x,x') < \delta.$
%The map $T$ is called lower semicontinuous if it is lower semicontinuous at every $x \in X.$ 
%\item $T$ is said to be closed if the graph of $T$ defined by $Gr(T):=\{(\boldsymbol{x}):y \in T(x)\}$ is a closed subset of $X \times Y$.
%\end{enumerate}
%\end{definition}

\begin{theorem}
The multi-valued function $\Phi:[\dim(X),\dim(X)+\dim(Y)] \rightarrow \mathcal{C}(X,Y)$ defined by 
\[
\Phi(\beta) :=\{f \in \mathcal{C}(X,Y): \dim (Gr(f)) =\beta \}
\] 
is lower semicontinuous. 
\end{theorem}
\begin{proof}
Let $U$ be an open set of $\mathcal{C}(X,Y).$ 
 In the light of Theorem \ref{densethm}, that is, $\Phi(\alpha)=\mathcal{S}_{\alpha}$ is a dense subset of $\mathcal{C}(X,Y)$, we obtain 
 \[
 \mathcal{S}(\alpha) \cap U \ne \emptyset, ~\forall ~ \alpha \in [\dim(X),\dim(X)+\dim(Y)].
 \]
 Now, by the very definition of lower semicontinuous, the result follows.
\end{proof}

\begin{remark}
Note that the multivalued mapping $\Phi$ is not closed. To show this, let $f \in \mathcal{C}(X,Y)$ with $\dim(Gr(f)) > \dim(X).$ Consider a sequence of Lipschitz functions $(f_k)$ converging to $f$ uniformly. It is obvious that $\dim(Gr(f_k))= \dim(X).$ Now, we have $\big(\dim(X),f_k\big) \to \big(\dim(X),f\big)$ as $n \to \infty.$ Using $\big(\dim(X), f_k\big) \in Gr(\Phi)$ and $\big(\dim(X), f_k\big) \rightarrow \big(\dim(X),f\big)$ with $\dim(Gr(f)) > \dim(X),$ we get the result.   
\end{remark}

   \section{conclusion}   
  This paper has been intended to develop a newly defined notion of constrained approximation termed as dimension preserving approximation for bivariate functions. The later work of the paper has introduced some multi-valued operators associated with bivariate $\alpha$-fractal functions. The notion of dimension preserving approximation is new, and demands further developments. In particular, dimension preserving approximation of set-valued mappings may be one of our future investigations.
 %\subsection*{Acknowledgements}

\bibliographystyle{amsplain}

\begin{thebibliography}{10}
\bibitem {Aubin} J. P. Aubin, H. Frankowska, Set-Valued Analysis, Birkh\"{a}user, Boston, 1990.
\bibitem {MF1} M. F. Barnsley, Fractal functions and interpolation, Constr. Approx. 2 (1986) 303-329.
%\bibitem {MF5}  M. F. Barnsley, Fractal Everywhere,
% Academic Press, Orlando, Florida, 1988.
%\bibitem{MF2} M. F. Barnsley, A.N. Harrington, The calculus of fractal interpolation functions, J. Approx. Theory
%57 (1989) 14-34.
\bibitem{MF6} M. F. Barnsley, J. Elton, D. P. Hardin and P. R. Massopust,
Hidden variable fractal interpolation functions, SIAM J. Math.
Anal. 20(5) (1989) 1218-1248.
\bibitem{MF4} M. F. Barnsley, P. R. Massopust, Bilinear fractal interpolation and box dimension, J. Approx. Theory 192 (2015) 362-378.
\bibitem{CC} P. G. Cazassa, O. Christensen, Perturbation of operators and application to
	frame theory, J. Fourier Anal. Appl. 3(5) (1997) 543-557.
	
 \bibitem{DS} F. Deustch, I. Singar, On single-valuedness of convex set-valued maps, Set-Valued Var Anal. 1 (1993) 97-103.
\bibitem {Devore} R. DeVore, One sided approximation of functions, J. Approx. Theory, 1 (1968) 11-25.
\bibitem{Falc2} K. J. Falconer, The Hausdorff dimension of self-affine fractals, Math. Proc. Camb. Phil. Soc. 103 (1988) 339-350.
%\bibitem{Falc3} K. J. Falconer, Techniques in Fractal Geometry, John Wily \& Sons, New York, 1997
\bibitem {Fal} K. J. Falconer, Fractal Geometry: Mathematical Foundations and Applications, John Wiley \& Sons Inc., New York, 1999.
\bibitem {Fraser} K. J. Falconer, J. M. Fraser, The horizon problem for  prevalent surfaces, Math. Proc. Camb. Phil. Soc. (2011), 151, 355.

\bibitem {Gal}  S. G. Gal, Shape preserving approximation by real and complex polynomials, Birkh\"auser, Boston, Mass, USA 2008.
\bibitem {Hardin} D. P. Hardin, P. R. Massopust, The capacity for a class of fractal functions, Commun. Math. Phys. 105 (1986) 455-460.
\bibitem{HM} D. P. Hardin, P. R. Massopust, Fractal interpolation functions from $\R^n$ to $\R^m$ and their projections, Zeitschrift f\"ur Analysis u. i. Anw. 12 (1993), 535-548.
\bibitem{H} J. Hutchinson, Fractals and self-similarity, Indiana Univ. Math. J. 30 (1981) 713-747.
\bibitem {PM1} P. R. Massopust, Fractal Functions, Fractal Surfaces, and Wavelets. 2nd ed., Academic Press, San Diego, 2016.
\bibitem {RD} R. D. Mauldin, S. C. Williams, On the Hausdorff dimension of some graphs, Trans. Amer. Math. Soc. 298 (1986) 793-803. 
\bibitem {M2} M. A. Navascu\'es, Fractal polynomial interpolation, Z. Anal. Anwend. 25(2) (2005) 401-418.
\bibitem {M1} M. A. Navascu\'es, Fractal approximation, Complex Anal. Oper. Theory, 4(4) (2010) 953-974.
\bibitem{Ruan} H.-J. Ruan and Q. Xu, Fractal interpolation surfaces on Rectangular Grids, Bull. Aust. Math. Soc. 91 (2015) 435-446.
\bibitem {Rudin} W. Rudin, Principles of Mathematical Analysis, 3rd Edition, McGraw-Hill, New York, 1976.
\bibitem {Shen} W. Shen, Hausdorff dimension of the graphs of the classical Weierstrass functions, Math. Z. 289 (2018) 223-266. 
\bibitem {Totik} V. Totik, Approximation by Bernstein polynomials, Amer. J. Math. 114(4) (1994) 995-1018.
\bibitem{VM} S. Verma, P. R. Massopust, Dimension preserving approximation, arXiv:2002.05061, Feb 2020.
\bibitem{VV1} S. Verma, P. Viswanathan, A Fractal Operator Associated with Bivariate Fractal Interpolation Functions on Rectangular Grids, Results Math 75, 28 (2020).
\bibitem{VV2} S. Verma, P. Viswanathan, Parameter identification for a class of bivariate fractal interpolation functions and constrained approximation, Numer. Fun. Anal. Opt. 41(9) (2020) 1109-1148.
\bibitem{VV3} S. Verma, P. Viswanathan, A fractalization of rational trigonometric functions, Mediterr. J. Math., 17:93(2020).

\end{thebibliography}

\end{document}